\documentclass{amsart}

\usepackage{ulem}
\usepackage{multicol,latexsym}
\usepackage{textcomp}
\usepackage[ansinew]{inputenc}
\usepackage{amssymb}
\usepackage{bigfoot}
\usepackage{graphicx}
\usepackage{bm}
\usepackage{color}
\usepackage{natbib}
\usepackage{mathdots}
\usepackage{multicol}
\usepackage{amssymb, amsmath, amsthm}
\usepackage{bm}
\theoremstyle{plain}
\usepackage{graphicx}
\usepackage[aboveskip=-5pt,position=bottom]{caption}
\usepackage{url}
\usepackage{hyperref}

\newtheorem{lemma}{Lemma}[section]

\newtheorem{theorem}[lemma]{Theorem}
\newtheorem{cor}[lemma]{Corollary}
\newtheorem{prop}[lemma]{Proposition}
\newtheorem{exam}[lemma]{\normalfont \scshape
 Example}
\newtheorem{rem}[lemma]{\normalfont \scshape Remark}

\newcommand{\R}{\mathbb{R}}
\newcommand{\N}{\mathbb{N}}

\newcommand{\norm}[1]{\left\Vert#1\right\Vert}
\newcommand{\abs}[1]{\left\vert#1\right\vert}
\newcommand{\set}[1]{\left\{#1\right\}}

\newcommand{\bfx}{\bm{x}}
\newcommand{\bfzero}{\bm{0}}
\newcommand{\bfinfty}{\bm{\infty}}

\newcommand{\bfone}{\bm{1}}

\newcommand{\bfc}{\bm{c}}

\newcommand{\bfe}{\bm{e}}
\newcommand{\bfk}{\bm{k}}

\newcommand{\bfU}{\bm{U}}
\newcommand{\bfu}{\bm{u}}

\newcommand{\bfX}{\bm{X}}
\newcommand{\bfY}{\bm{Y}}

\newcommand{\bfZ}{\bm{Z}}

\newcommand{\bfxi}{\bm{\xi}}

\allowdisplaybreaks[4]
\begin{document}

\title[Multivariate Order Statistics]{Multivariate Order Statistics:\linebreak The Intermediate Case}
\author{Michael Falk and Florian Wisheckel}
\address{University of W\"{u}rzburg,
Institute of Mathematics,  Emil-Fischer-Str. 30, 97074 W\"{u}rzburg, Germany.}
\email{michael.falk@uni-wuerzburg.de, florian.wisheckel@uni-wuerzburg.de}

\subjclass[2010]{Primary 62G30, secondary 62H10}%
\keywords{Multivariate order statistics, intermediate order statistics, copula, domain of attraction, $D$-norm, von Mises type conditions, asymptotic normality}%


\begin{abstract}
Asymptotic normality of intermediate order statistics taken from univariate iid random variables is well-known. We generalize this result to random vectors in arbitrary dimension, where the order statistics are taken componentwise.
\end{abstract}

\maketitle

\section{Introduction}
Let $\bfX^{(1)}=\left(X_1^{(1)},\dots,X_d^{(1)}\right),\dots, \bfX^{(n)}=\left(X_1^{(n)},\dots,X_d^{(n)}\right)$ be independent copies of a random vector (rv) $\bfX=(X_1,\dots,X_d)$ that realizes in $\R^d$. By
\[
X_{1:n,i}\le X_{2:n,i}\le \dots \le X_{n:n,i}
\]
we denote the ordered values of the $i$-th components of $\bfX^{(1)},\dots,\bfX^{(n)}$, $1\le i\le d$. Then $\left(X_{j_1:n,1},\dots,X_{j_d:n,d}\right)$ with $1\le j_1,\dots,j_d\le n$, is a rv of order statistics (os) in each component. We call it a \emph{multivariate} os.

The univariate case $d=1$ is, clearly, well investigated; standard references are the books by \citet{david81}, \citet{reiss89}, \citet{gal87}, \citet{davna04}, \citet{arnbn08}, among others. In the multivariate case $d\ge 2$, the focus has been on the investigation of the rv of componentwise maxima $\left(X_{n:n,1},\dots,X_{n:n,d}\right)$ (\citet{balr77}, \citet{dehar77}, \citet{resn87}, \citet{vatan85}, \citet{beirgotese04}, \citet{dehaf06}, \citet{fahure10}, among others).

Much less is known in the extremal case $\left(X_{n-k_1:n,1},\dots,X_{n-k_d:n}\right)$ with $k_1,\dots,k_d\in\N$ fixed; one reference is \citet{gal75}. Asymptotic normality of the random vector $\left(X_{j_1:n,1},\dots,X_{j_d:n,d}\right)$ in the case of central os is established in \citet[Theorem 7.1.2]{reiss89}. In this case the indices $j_i=j_i(n)$ depend on $n$ and have to satisfy $j_i(n)/n\to_{n\to\infty}q_i\in (0,1)$, $1\le i\le d$.

In the case of intermediate os we require $j_i=j_i(n)=n-k_i$, where $k_i=k_i(n)\to_{n\to\infty}\infty$ with $k_i/n\to_{n\to\infty} 0$. Asymptotic normality of intermediate os in the univariate case under fairly general von Mises conditions was established in \citet{falk89}. \citet{balh78a} and \citet[Theorem 7.1]{balh78b} proved that for particular underlying distribution function (df) $F$, $X_{n-k+1:n}$ may have \emph{any} limiting distribution if it is suitably standardized and if the sequence $k$ is chosen appropriately.

As pointed out by \citet{smir67}, a (nondegenerate) limiting distribution of $X_{n-k+1:n}$ different from the normal one can only occur if $k$ has an \emph{exact} preassigned asymptotic behavior. Assuming only $k\to_{n\to\infty}\infty$, $k/n\to_{n\to\infty}0$, \citet{smir67} gave necessary and sufficient conditions for $F$ such that $X_{n-k+1:n}$ is asymptotically normal, and he specified the appropriate norming constants, see condition \eqref{eqn:Smirnov's_condition} below.

Smirnov's result was extended to multivariate intermediate os by \citet{chedehya97}. They identify the class of limiting distributions of $\left(X_{n-k_1:n,1},\dots,X_{n-k_d:n,d}\right)$ after suitable normalizing and centering, and gave necessary and sufficient conditions of weak convergence.

\citet{cooil85} established multivariate extensions of the univariate case by considering vectors of intermediate os $\left(X_{n-k_1+1:n},\dots,X_{n-k_d+1:n}\right)$ taken from the same sample of univariate os $X_{1:n}\le\dots\le X_{n:n}$ but with pairwise different $k_1,\dots,k_d$. \citet{barakat01} investigates the limit distribution of bivariate os in all nine possible combinations of central, intermediate and extreme os.

According to (\citet{sklar59, sklar96}), the df of $\bfX=(X_1,\dots,X_d)$ can be decomposed into a copula and the df $F_i$ of each component $X_i$, $1\le i\le d$. We will establish in this paper asymptotic normality of the vector of multivariate os $\left(X_{n-k_1:n,1},\dots,X_{n-k_d:n,d}\right)$ in the intermediate case. This is achieved under the condition that the copula corresponding to $\bfX$ is in the max-domain of attraction of a multivariate extreme value df together with the assumption that each univariate marginal df $F_i$ satisfies a von Mises condition and that the norming constants satisfy Smirnov's condition \eqref{eqn:Smirnov's_condition} below.

\section{Main Results: Copula Case}
We consider first the case that the df of the rv $\bfX$ is a copula, $C$ say, on $\R^d$. We require that $C$ is in the max-domain of attraction of a non-degenerate multivariate extreme-value df (evd) $G$, i.e.
\begin{equation}\label{eqn:max-domain-of-attraction_for_copula}
C^n\left(\bfone+\frac {\bfx} n\right)\to_{n\to\infty} G(\bfx),\qquad \bfx\in\R^d,
\end{equation}
where $\bfone=(1,\dots,1)\in\R^d$ and all operations on vectors are meant componentwise.
In this case, there exists a $D$-norm $\norm{\cdot}_D$ on $\R^d$ such that
\begin{equation}\label{eqn:representation_of_standard_evd}
G(\bfx) =\exp\left(-\norm{\bfx}_D\right),\qquad \bfx\le\bfzero\in\R^d.
\end{equation}
A common norm $\norm\cdot$ on $\R^d$ is a $D$-norm $\norm\cdot_D$, if there exists a rv $\bfZ=(Z_1,\dots,Z_d)$ on $\R^d$ with the two properties $Z_i\ge 0$, $E(Z_i)=1$ for $i=1,\dots,d$, such that
\[
\norm{\bfx}_D=E\left(\max_{1\le i\le d}\abs{x_i}Z_i\right),\qquad \bfx\in\R^d.
\]
The rv $\bfZ$ is called a generator of the $D$-norm, and we add the index $D$ to the norm symbol, meaning dependence.

Representation \eqref{eqn:representation_of_standard_evd} is just a reformulation of the Pickands-de Haan-Resnick-Vatan characterization of a multivariate evd, using $D$-norms; see, e.g. \citet[Chapter 4]{fahure10}. Examples of $D$-norms are the sup-norm $\norm{\bfx}_\infty=\max_{1\le i\le d}\abs{x_i}$ as well as the complete logistic family $\norm{\bfx}_p=\left(\sum_{i=1}^d\abs{x_i}^p\right)^{1/p}$, $p\ge 1$. For a systematic treatment of $D$-norms we refer to the booklet by Falk (2016).\footnote{\verb|http://www.statistik-mathematik.uni-wuerzburg.de/fileadmin/10040800/D-norms-tutorial_book.pdf|}

A straightforward analysis shows that \eqref{eqn:max-domain-of-attraction_for_copula} \& \eqref{eqn:representation_of_standard_evd} are equivalent with the condition that there exists a $D$-norm on $\R^d$ such that
\begin{equation}\label{eqn:D-norm_expansion_of_copula}
C(\bfu)=1-\norm{\bfone-\bfu}_D+ o\left(\norm{\bfone-\bfu}\right)
\end{equation}
as $\bfu\to\bfone$, uniformly for $\bfu\in[0,1]^d$.

We are now ready to state asymptotic normality of the vector of multivariate os in the intermediate case with underlying copula. By $\bfe_j:=(0,\dots,0,1,0,\dots,0)\in\R^d$ with denote the $j$-th unit vector, $j=1,\dots,d$.

\begin{theorem}[The Copula Case]\label{theo:main_result_for_copula}
Suppose that the rv $\bfX=(X_1,\dots,X_d)$ follows a copula $C$, which satisfies expansion \eqref{eqn:D-norm_expansion_of_copula} with some $D$-norm $\norm\cdot_D$ on $\R^d$. Let $\bfk=\bfk(n)=(k_1,\dots,k_d)\in\set{1,\dots,n-1}$, $n\in\N$, satisfy $k_i/k_j\to k_{ij}^2\in (0,\infty)$ for all pairs of components $1\le i,j\le d$, $\norm{\bfk}\to\infty$ and $\norm{\bfk}/n\to 0$ as $n\to\infty$. Then the rv of componentwise intermediate os is asymptotically normal:
\[
\left(\frac n{\sqrt {k_i}}\left(X_{n-k_i:n,i}-\frac{n-k_i}n\right)  \right)_{i=1}^d \to_D N\left(\bfzero, \Sigma\right),
\]
where the $d\times d$-covariance matrix is given by
\[
\Sigma=(\sigma_{ij})=\begin{cases}
1,&\mbox{\upshape if }i=j\\
k_{ij}+k_{ji} - \norm{k_{ij}\bfe_i+k_{ji}\bfe_j}_D,&\mbox{\upshape if }i\not=j.
\end{cases}
\]
\end{theorem}

If, for example, $\norm{\bfx}_D=\norm{\bfx}_p=\left(\sum_{i=1}^p\abs{x_i}^p\right)^{1/p}$, $p\ge 1$, then $\sigma_{ij}= k_{ij}+k_{ji} - \left(k_{ij}^p + k_{ji}^p\right)^{1/p}$, $i\not=j$.

\begin{rem}\upshape
Note that $\sigma_{ij}=0$, $i\not=j$, if $\norm\cdot_D=\norm\cdot_1$, which is the case if the margins of $G(\bfx)=\exp(-\norm{\bfx}_D)=\prod_{i=1}^d\exp(x_i)$, $\bfx\le\bfzero\in\R^d$, are independent. Then the components of $\bfX=(X_1,\dots,X_d)$ are called tail independent. The reverse implication is true as well, i.e., the preceding result entails that the componentwise intermediate os $X_{n-k_1:n,1},\dots,X_{n-k_d:n,d}$ are asymptotically independent if, and only if, they are pairwise asymptotically independent. But this is equivalent with the condition that the $\norm{\cdot}_D=\norm{\cdot}_1$, see Section 1.3 in Falk (2016).

Note that $\sigma_{ij}\ge 0$ for each pair $i,j$, i.e., the componentwise os are asymptotically positively correlated. This follows from the usual triangular inequality, satisfied by each norm, and the fact that a $D$-norm is in general standardized, i.e., $\norm{\bfe_j}_D=1$, $1\le j\le d$.
\end{rem}

\begin{cor}\label{cor:asymptotic_normality_of_intermediate_os}
If we choose identical $k_i$ in the preceding result, i.e. $k_1=\dots=k_d=k$, then we obtain under the conditions of Theorem \ref{theo:main_result_for_copula}
\[
\frac n{\sqrt k}\left(X_{n-k:n,i}- \frac{n-k}n\right)_{i=1}^d \to_D N(\bfzero,\Sigma)
\]
with
\[
\Sigma=(\sigma_{ij})=\begin{cases}
1,&\mbox{\upshape if }i=j\\
2- \norm{\bfe_i+\bfe_j}_D,&\mbox{\upshape if }i\not=j.
\end{cases}
\]
\end{cor}

Let $U_{1:n}\le U_{2:n}\le \dots\le U_{n:n}$ denote the os of $n$ independent and uniformly on $(0,1)$ distributed rv $U_1,\dots,U_n$. It is well-known that
\[
\left(U_{i:n}\right)_{i=1}^n =_D \left(\frac{\sum_{j=1}^i \eta_j}{\sum_{j=1}^{n+1}\eta_j}\right)_{i=1}^n,
\]
where $\eta_1,\dots,\eta_{n+1}$ are iid standard exponential rv; see, e.g., \citet[Corollary 1.6.9]{reiss89}.

Let $\xi_1,\xi_2,\dots,\xi_{2(n+1)}$ be iid standard normal distributed rv. From the fact that $\left(\xi_1^2+\xi_2^2\right)/2$ follows the standard exponential distribution on $(0,\infty)$, we thus obtain (\citet[Problem 1.17]{reiss89}) the representation
\begin{equation}\label{eqn:representation_of_os_via_normal_rv}
\left(U_{i:n}\right)_{i=1}^n =_D \left(\frac{\sum_{j=1}^{2i} \xi_j^2}{\sum_{j=1}^{2(n+1)}\xi_j^2}\right)_{i=1}^n.
\end{equation}
Corollary \ref{cor:asymptotic_normality_of_intermediate_os} now opens a way to tackle at least partially and asymptotically a multivariate extension of the above representation \eqref{eqn:representation_of_os_via_normal_rv}.

\begin{cor}
Suppose that the $d\times d$-matrix $\Lambda$ with entries
\[
\lambda_{ij}=\sigma_{ij}^{1/2}=\begin{cases}
1,&\mbox{\upshape if }i=j\\
\left(2-\norm{\bfe_i+\bfe_j}_D\right)^{1/2}, &\mbox{\upshape if } i\not=j
\end{cases}
\]
is positive semidefinite and let $\bfxi^{(1)},\bfxi^{(2)},\dots$ be independent copies of the random vector $\bfxi=(\xi_1,\dots,\xi_d)$, which follows the normal distribution $N(\bfzero,\Lambda)$ on $\R^d$.
Then we obtain under the conditions of Corollary \ref{cor:asymptotic_normality_of_intermediate_os}
\[
\sup_{\bfx\in\R^d}\abs{P\left(\left(X_{n-k:n,i}\right)_{i=1}^d\le \bfx \right) - P\left(\left(\frac{\sum_{j=1}^{2(n-k)}\xi_i^{(j)^2}}{\sum_{j=1}^{2(n+1)}\xi_i^{(j)^2}} \right)_{i=1}^d\le\bfx \right)} \to_{n\to\infty} 0.
\]
\end{cor}

Note that the univariate marginal distributions in the above result coincide due to equation \eqref{eqn:representation_of_os_via_normal_rv}. If a matrix is positive semidefinite with nonnegative entries, the matrix of the square roots of its entries is not necessarily semidefinite again. Take, for example, the $3\times 3$-matrix  with rows $1,0,a|0, 1,a|a,a,1$. This matrix is positive definite for $a=3^{-1/2}$, but not for $a=3^{-1/4}$. The matrix $\Lambda$ is positive semidefinite, if the value of $\norm{\bfe_i+\bfe_j}_D$ does not depend on the pair $i\not = j$, in which case $\Lambda$ satisfies the \emph{compound symmetry condition}.

\begin{proof}
From Corollary \ref{cor:asymptotic_normality_of_intermediate_os} we obtain that
\[
\frac n{\sqrt k}\left(X_{n-k:n,i}-\frac{n-k}n\right)_{i=1}^d \to_D N(\bfzero,\Sigma).
\]
The assertion follows, if we establish
\[
\frac n{\sqrt k}\left(\frac{\sum_{j=1}^{2(n-k)}\xi_i^{(j)^2}}{\sum_{j=1}^{2(n+1)}\xi_i^{(j)^2}} - \frac{n-k}n  \right)_{i=1}^d \to_D N(\bfzero,\Sigma)
\]
as well. But this follows from the central limit theorem and elementary arguments, using the fact that $\mbox{Cov}(X^2,Y^2)=2c^2$, if $(X,Y)$ is bivariate normal with $\mbox{Cov}(X,Y)=c$.
\end{proof}

The proof of Theorem \ref{theo:main_result_for_copula} requires a suitable multivariate central limit theorem for arrays. To ease its reference we state it explicitly here. It follows from the univariate version based on Lindeberg's condition, see, e.g., \citet{billi12}, together with the Cram\'{e}r-Wold device. Recall that all operations on vectors are meant componentwise.

\begin{lemma}[Multivariate Central Limit Theorem for Arrays]\label{lem:multivariate_central_limit_theorem}
Let $\bfX^{(1)}_n,\dots,\bfX^{(n)}_n$ be iid rv for each $n\in\N$, bounded by some constant $\bfc=(c_1,\dots,c_d)>\bfzero\in\R^d$ and with mean zero. Suppose there is a sequence $\bfc^{(n)}\in\R^d$ with $nc_i^{(n)}\to_{n\to\infty}\infty$ for $i=1,\dots,d$, such that $\mathrm{Cov}\left(\bfX_n ^{(1)}\right)=C^{(n)}\Sigma^{(n)}C^{(n)}$, $n\in\N$, where $C^{(n)}=\mathrm{diag}\left(\sqrt{\bfc^{(n)}}\right)$ and $\Sigma^{(n)}\to_{n\to\infty}\Sigma$. Then
\[
\frac 1 {\sqrt{n\bfc^{(n)}}} \sum_{i=1}^n\bfX^{(i)}_n\to_D N(\bfzero,\Sigma).
\]
\end{lemma}

\begin{proof}[Proof of Theorem \ref{theo:main_result_for_copula}]
Choose $\bfx=(x_1,\dots,x_d)\in\R^d$. Elementary arguments yield
\begin{align}\label{eqn:representation_of_probability}
&P\left(\left(\frac n{\sqrt {k_i}}\left(X_{n-k_i:n,i} - \frac{n-k_i}n \right) \right)_{i=1}^d \le\bfx \right)\nonumber\\
&= P\left(X_{n-k_i:n,i}\le \frac{\sqrt{k_i}}n x_i + \frac{n-k_i}n,\;1\le i\le d \right)\nonumber\\
&= P\left(\sum_{j=1}^n 1_{\left[0,\frac{\sqrt{k_i}}n x_i + \frac{n-k_i}n\right]}\left(X_i^{(j)}\right)\ge n-k_i,\;1\le i\le d \right)\nonumber\\
&= P\left(\left(\frac 1{\sqrt{k_i}}\sum_{j=1}^n\left(\frac{\sqrt{k_i}}n x_i + \frac{n-k_i}n - 1_{\left[0,\frac{\sqrt{k_i}}n x_i + \frac{n-k_i}n\right]}\left(X_i^{(j)}\right) \right) \right)_{i=1}^d \le\bfx \right).
\end{align}

Put now
\[
\bfY^{(n)}:= \left(Y_1^{(n)},\dots,Y_d^{(n)}\right) := \left( 1_{\left[0,\frac{\sqrt{k_i}}n x_i + \frac{n-k_i}n\right]}\left(X_i\right) \right)_{i=1}^d
\]
with values in $\set{0,1}^d$. The entries of its covariance matrix $\Sigma^{(n)}=\left(\sigma_{ij}^{(n)} \right)$ are for $i\not=j$ given by
\begin{align*}
\sigma_{ij}^{(n)} &= E\left(Y_i^{(n)}Y_j^{(n)}\right) - E\left(Y_i^{(n)}\right) E\left(Y_j^{(n)}\right)\\
&= P\left(Y_i^{(n)}=Y_j^{(n)}=1\right)- P\left(Y_i^{(n)}=1\right)P\left(Y_j^{(n)}=1\right)\\
&= P\left(X_i\le \frac{\sqrt{k_i}}n x_i + \frac{n-k_i}n,\, X_j\le \frac{\sqrt{k_j}}n x_j + \frac{n-k_j}n \right)\\
&\hspace*{2cm}- P\left(X_i\le \frac{\sqrt{k_i}}n x_i + \frac{n-k_i}n\right)  P\left(X_j\le \frac{\sqrt{k_j}}n x_j + \frac{n-k_j}n\right)\\
&= C_{ij}\left(\frac{\sqrt{k_i}}n x_i + \frac{n-k_i}n,\,\frac{\sqrt{k_j}}n x_j + \frac{n-k_j}n \right)\\
&\hspace*{2cm}- \left( \frac{\sqrt{k_i}}n x_i + \frac{n-k_i}n\right) \left( \frac{\sqrt{k_j}}n x_j + \frac{n-k_j}n\right)
\end{align*}
if $n$ is large, where
\[
C_{ij}(u,v):= C\left(u\bfe_i + v\bfe_j + \sum_{1\le m\le d,\,m\not= i,j}\bfe_m\right),\qquad u,v\in[0,1].
\]
Expansion \eqref{eqn:D-norm_expansion_of_copula} now implies in case $i\not=j$
\begin{align*}
\sigma_{ij}^{(n)} &= 1 - \norm{\left(\frac{k_i}n - \frac{\sqrt{k_i}}n x_i \right)\bfe_i + \left(\frac{k_j}n - \frac{\sqrt{k_j}}n x_j \right)\bfe_j }_D + o\left(\frac{\sqrt{k_ik_j}}n\right)\\
&\hspace*{2cm}- \left(\frac{\sqrt{k_i}}n x_i - \frac{k_i}n + 1\right) \left(\frac{\sqrt{k_j}}n x_j - \frac{k_j}n + 1\right)\\
&= - \norm{\left(\frac{k_i}n - \frac{\sqrt{k_i}}n x_i \right)\bfe_i + \left(\frac{k_j}n - \frac{\sqrt{k_j}}n x_j \right)\bfe_j }_D + \frac{k_i+k_j}n + o\left(\frac{\sqrt{k_ik_j}}n\right)\\
&= \frac{\sqrt{k_ik_j}}n \left(k_{ij}+k_{ji}- \norm{k_{ij}\bfe_i+k_{ji}\bfe_j}_D + o(1) \right).
\end{align*}

For $i=j$ one deduces
\[
\sigma_{ii}^{(n)} = \frac{k_i}n (1+o(1)).
\]

The asymptotic normality $N(\bfzero,\Sigma)(-\bfinfty,\bfx]$ of the final term in equation \eqref{eqn:representation_of_probability} now follows from Lemma \ref{lem:multivariate_central_limit_theorem}.

\end{proof}

\section{Main Results: General Case}
Let $F$ be a df on $\R^d$ with univariate margins $F_1,\dots,F_d$. From Sklar's theorem (\citet{sklar59, sklar96}) we know that there exists a copula $C$ on $\R^d$ such that $F(\bfx)= C(F_1(x_1),\dots,F_d(x_d))$ for each $\bfx=(x_1,\dots,x_d)\in\R^d$.

Let $\bfX^{(1)},\bfX^{(2)},\dots$ be independent copies of the random vector $\bfX$, which follows this df $F$. We can assume the representation
\[
\bfX=\left(F_1^{-1}(U_1),\dots,F_d^{-1}(U_d)\right),
\]
where $\bfU=(U_1,\dots,U_d)$ follows the copula $C$ and $F_i^{-1}(u):=\inf\set{t\in\R:\,F_i(t)\ge u}$, $u\in(0,1)$, is the generalized inverse of $F_i$, $1\le i\le d$. Equally, we can assume the representation
\[
\bfX^{(j)}=\left(F_1^{-1}\left(U_1^{(j)}\right),\dots, F_d^{(-1)}\left(U_d^{(j)}\right)\right),\qquad j=1,2,\dots
\]
where $\bfU^{(1)},\bfU^{(2)},\ldots$ are independent copies of $\bfU$.

Put $\omega(F_i):=\sup\set{x\in\R:\,F_i(x)< 1}\in(-\infty,\infty]$, the upper endpoint of the support of $F_i$, and suppose that the derivative $F_i'=f_i$ exists and is positive throughout some left neighborhood of $\omega(F_i)$.  Let $k_i=k_i(n)\in\set{1,\dots,n}$ satisfy $k_i\to_{n\to\infty}\infty$, $k_i/n\to_{n\to\infty}0$.  It follows from \citet[Theorem 2.1]{falk89} that under appropriate von Mises type conditions on $F_i$ stated below
\[
\frac{X_{n-k_i+1:n,i}-d_{ni}}{c_{ni}}\to_D N(0,1)
\]
for any sequences $c_{ni}>0$, $d_{ni}\in\R$, which satisfy
\begin{equation}\label{eqn:condition_on_norming_constants}
\lim_{n\to\infty}\frac{c_{ni}}{a_{ni}}=1\quad\mbox{ and }\quad \lim_{n\to\infty}\frac{d_{ni}-b_{ni}}{a_{ni}}=0,
\end{equation}
where
\[
b_{ni}:=F_i^{-1}\left(1-\frac{k_i}n\right),\quad a_{ni}:= \frac{k_i^{1/2}}{nf_i(b_{ni})},\qquad 1\le i\le d.
\]

Theorem 1 of \citet{smir67} shows that the distribution of $c_n^{-1}(X_{n-k_i+1:n}-d_n)$ converges weakly to $N(0,1)$ for \emph{some} choice of constants $c_n>0$, $d_n\in\R$, if and only if for any $x\in\R$
\begin{equation}\label{eqn:Smirnov's_condition}
\lim_{n\to\infty} \frac{k_i+n(F_i(c_nx+d_n)-1)}{k_i^{1/2}} = x.
\end{equation}

Next we state the three von Mises type conditions, under which we have asymptotic normality for intermediate multivariate os in the general case:

\noindent
$\omega(F_i)\in(-\infty,\infty]$ and
\begin{equation}\label{cond:von_Mises_1}
\lim_{x\uparrow \omega(F_i)} \frac{f_i(x) \int_x^{\omega(F_i)}1-F_i(t)\,dt} {(1-F_i(x))^2} =1,\tag{von Mises (1)}
\end{equation}\label{cond:von_Mises_2}
$\omega(F_i)=\infty$ and there exists $\alpha_i >0$ such that
\begin{equation}
\lim_{x\to\infty} \frac{xf_i(x)}{1-F_i(x)}=\alpha_i, \tag{von Mises (2)}
\end{equation}
$\omega<\infty$ and there exists $\alpha>0$ such that
\begin{equation}\label{cond:von Mises_3}
\lim_{x\uparrow \omega(F_i)} \frac{(\omega(F_i)-x)f_i(x)}{1-F_i(x)} =\alpha_i. \tag{von Mises (3)}
\end{equation}
The standard normal df as well as the df of the standard exponential df satisfy condition (1); the standard Pareto df $F_\alpha(x)$, $x\ge 1$, $\alpha>0$, satisfies condition (2) and the triangular df on $(-1,1)$ with density $f(x)=1-\abs x$, $x\in (-1,1)$, satisfies condition (3) with $\alpha=2$, for example.
For a discussion of these well studied and general conditions, each of which ensures that $F_i$ is in the domain of attraction of a univariate EVD, see, e.g. \citet{falk89}.

The following generalization of Theorem \ref{theo:main_result_for_copula} can now easily be established.

\begin{prop}
Suppose that the copula $C$ of $F$ satisfies condition \eqref{eqn:D-norm_expansion_of_copula}, i.e., $C$ is in the max-domain of attraction of a nondegenerate multivariate EVD, and suppose that each univariate margin $F_i$ of $F$ satisfies one of the von Mises type conditions (1), (2) or (3).

Let $\bfk=\bfk^{(n)}\in\set{1,\dots,n}^d$, $n\in\N$ satisfy $k_i/k_j\to_{n\to\infty}k_{ij}^2$ for all pairs of components $i,j=1,\dots,d$, $\norm{\bfk}\to_{n\to\infty}\infty$ and $\norm{\bfk}/n\to_{n\to\infty}0.$ Then the vector of intermediate multivariate os satisfies
\[
\left(\frac{X_{n-k_i+1:n,i}-d_{ni}}{c_{ni}}\right)_{i=1}^d \to_D N\left(\bfzero, \bm\Sigma\right)
\]
with $\bm\Sigma$ as in Theorem \ref{theo:main_result_for_copula} for any sequences $c_{ni}>0$, $d_{ni}\in\R$ which satisfy \eqref{eqn:condition_on_norming_constants}.
\end{prop}

\begin{proof}
We have for $\bfx=(x_1,\dots,x_d)\in\R^d$
\begin{align*}
&P\left(\frac{X_{n-k_i+1:n,i}-d_{ni}}{c_{ni}}\le x_i,\,1\le i\le d\right)\\
&= P\left(F_i^{-1}\left(U_{n-k_i+1:n,i}\right)\le c_{ni}x_i + d_{ni},\,1\le i\le d\right)\\
&= P\left(\frac n{k_i^{1/2}}\left(U_{n-k_i+1:n,i} - \frac{n-k_i}n \right)\le \frac{k_i + n\left(F_i\left(c_{ni}x_i+ d_{ni}\right)-1\right)}{k_i^{1/2}} \right).
\end{align*}
The assertion is now immediate from Theorem \ref{theo:main_result_for_copula} and Smirnov's condition \eqref{eqn:Smirnov's_condition}.
\end{proof}

\bibliographystyle{enbib_arXiv}
\bibliography{evt}

\end{document}